\documentclass[11pt,letterpapert]{article}
\usepackage{graphicx}
\usepackage{amssymb}
\usepackage{epstopdf}
\usepackage{fancyhdr}
\usepackage{amsmath}
\usepackage{amsthm}
\usepackage{hyperref}
\usepackage{makeidx}
\usepackage{blkarray}
\usepackage{mathdesign}
\usepackage{color}
\usepackage{geometry}
\usepackage{soul}

\definecolor{purple}{rgb}{.9,0,.9}

\makeatletter
\let\orgdescriptionlabel\descriptionlabel
\renewcommand*{\descriptionlabel}[1]{%
  \let\orglabel\label
  \let\label\@gobble
  \phantomsection
  \edef\@currentlabel{#1}%
  \let\label\orglabel
  \orgdescriptionlabel{#1}%
}
\makeatother

\newcommand{\Real}{\mathbb{R}}

\newcommand{\rfa}{\quad {\rm for \ all}\ }

\newcommand{\cA}{{\cal A}}\newcommand{\cB}{{\cal B}}
\newcommand{\cE}{{\cal E}}
\newcommand{\cH}{{\cal H}}
\newcommand{\cL}{{\cal L}}
\newcommand{\cM}{{\cal M}}\newcommand{\cN}{{\cal N}}\newcommand{\cO}{{\cal O}}
\newcommand{\cR}{{\cal R}}
\newcommand{\cS}{{\cal S}}\newcommand{\cT}{{\cal T}}
\newcommand{\cW}{{\cal W}}

\newcommand{\ba}{{\bf a}}
\newcommand{\be}{{\bf e}}

\newcommand{\bn}{{\bf n}}

\newcommand{\bt}{{\bf t}}
\newcommand{\bv}{{\bf v}}\newcommand{\bw}{{\bf w}}
\newcommand{\bA}{{\bf A}}
\newcommand{\bB}{{\bf B}}\newcommand{\bC}{{\bf C}}\newcommand{\bD}{{\bf D}}

\newcommand{\bJ}{{\bf J}}

\newcommand{\bP}{{\bf P}}

\newcommand{\bnu}{\boldsymbol{\nu}}

\newtheorem{theorem}{Theorem}[section]
\newtheorem{lemma}[theorem]{Lemma}

\newcommand\Tstrut{\rule{0pt}{2.6ex}}         

\newfont{\ssmsam}{msam5 scaled 750}
\newcommand{\bbox}{\mbox{\ssmsam\char'03}}
\newcommand{\zzbox}{\raisebox{-0.25ex}{$\bbox$}}
\newcommand{\zbox}[1]{\overset{\zzbox}{#1}}

\newcommand{\beqn}{\begin{equation}}
\newcommand{\eeqn}{\end{equation}}

\newcommand{\bzero}{{\bf 0}}

\newcommand{\da}{\,\text{da}}
\newcommand{\dv}{\,\text{dv}}
\newcommand{\dl}{\,\text{dl}}

\newcommand{\trans}{{\scriptscriptstyle\mskip-1mu\top\mskip-2mu}}

\title{A transport theorem for nonconvecting open sets\\ on an embedded manifold}

\author{Brian Seguin}

\begin{document}

\maketitle


\begin{abstract}
Most transport theorems---that is, a formula for the rate of change of an integral in which both the integrand and domain of integration depend on time---involve domains that evolve according to a flow map.  Such domains are said to be convecting.  Here a transport theorem for nonconvecting domains evolving on an embedded manifold is established.  While the domain is not convecting, it is assumed that the boundary of the domain does evolve according to a flow map is some generalized sense. The proof relies on considering the evolving set as a fixed set in one higher dimension and then using the divergence theorem.  The domains considered can be irregular in the sense that their boundaries need only be Lipschitz.  Tools from geometric measure theory are used to deal with this irregularity.
\end{abstract}

\section{Introduction}

One often encounters the problem of having to find a formula for the derivative of an integral with respect to a parameter in which both the integrand and the domain of integration depend on the parameter.  Such problems occur in continuum physics when trying to convert a global balance law to a local one \cite{GFA} or in shape optimization problems when computing the variation of a functional \cite{BB,Walker}.  The simplest case of this kind of result is the Leibniz integral rule, which states that if $\phi:\Real^2\rightarrow\Real$ and $a,b:\Real\rightarrow\Real$ are sufficiently smooth functions, then
\beqn
\frac{d}{dt}\int_{a(t)}^{b(t)} \phi(x,t)\, dx = \int_{a(t)}^{b(t)} \frac{\partial}{\partial t}\phi(x,t)\, dx+ \phi(b(t),t)b'(t)-\phi(a(t),t)a'(t).
\eeqn
In higher dimensions, such a formula is referred to as a transport theorem.  This is because the analogous result in three dimensions is called Reynolds' transport theorem, which involves an open, bounded set $\cR$ in $\Real^3$ with smooth boundary that is evolving according to a given flow map with velocity field $\bv$.
In this case, given a smooth function $\phi$ depending on position and time, one has 
\beqn\label{RTT}
\frac{d}{dt}\int_{\cR}\phi \dv = \int_{\cR} \frac{\partial}{\partial t}\phi \dv + \int_{\partial\cR}\phi \bv\cdot\bn \da,
\eeqn
where $\bn$ is the exterior unit-normal to $\cR$.  A domain evolving according to a flow map is called a convecting domain.

Extensions of the classical transport theorems of Leibniz and Reynolds have been established by considering different types of evolving domains.  A particularly important example of this being the surface transport theorem \cite{PFG,FT,GSW,Lidstrom}.  Here, one considers a convecting surface $\cS$ in three dimensions.  The resulting formula for the time derivative is
\beqn\label{STT}
\frac{d}{dt}\int_{\cS}\phi\da = \int_{\cS}(\zbox{\phi} - 2H\phi \bv\cdot\bn)\da + \int_{\partial\cS} \phi \bv\cdot\bnu\dl,
\eeqn
where $\zbox{\phi}$ is what is known as the normal time-derivative of $\phi$ and it measures how quickly $\phi$ is changing in the direction orthogonal to the surface, $H$ is the mean curvature of $\cS$, and $\bnu$ is the unit binormal that is tangent to $\cS$ but orthogonal to $\partial \cS$.  More generally, a transport theorem for a $k$-dimensional manifold convecting in an $n$-dimensional semi-Riemannian manifold was established by Betounes \cite{Betounes}.  The formula for the derivative in this case is similar to what appears in \eqref{STT} with the term involving the mean curvature being replaced by one that contains the mean-curvature vector.

The transport theorems mentioned so far hold for evolving smooth domains, but results for irregular domains have been obtained.  Falach and Segev established generalized transport theorems by modeling the domain of integration either as a de Rham current \cite{FS1} or a flat chain \cite{FS2} in the spirit of Federer's geometric measure theory \cite{Fed}.  In both cases, the domain was convecting according to a given flow map, though in \cite{FS2} this map was only required to be Lipschitz.  Seguin and Fried \cite{SF} proved a transport theorem involving irregular domains using Harrison's theory \cite{Harrison} of differential chains.  Besides allowing for irregular evolving domains, this transport theorem holds for domains that need not be convecting.  This allows for the domain to develop holes or transition from smooth to fractal during the course of its evolution.  For a not-so-technical description of this result, see \cite{SHF}.  


Of all of the previously mentioned results on transport theorems, only two did not involve convecting domains: the result of Seguin and Fried \cite{SF} and Gurtin, Struthers, and Williams' \cite{GSW} proof of the surface transport theorem.  The first of these has the drawback that showing a particular evolving domain can be modeled as a time-dependent differential chain with the properties necessary to apply the generalized transport theorem can be challenging in practice, while the second only holds for smooth surfaces.  Another method of proving a transport theorem for nonconvecting domains was provided by Estrada and Kanwal \cite{EK} using the theory of distributions, but their result involved smooth domains.  When considering a nonconvecting domain, there is no flow map, and hence no associated velocity field $\bv$.  In this case, the transport theorems take a slightly different form.  For example, in the Reynolds' transport theorem \eqref{RTT} for a nonconvecting domain the term $\bv\cdot\bn$ is replaced by a scalar normal velocity which describes the evolution of the boundary of $\cR$.  

Here we establish a transport theorem involving an evolving open set that lives on a fixed manifold embedded in a Euclidean space that is not convecting and can have certain types of irregularities.  Since the manifold on which the set is evolving is stationary, this transport theorem will take a form similar to \eqref{RTT}.  For a simple example of the type of domain that will be considered, look at the evolving domain shown in Figure~\ref{ExMD}.  The domain initially occupies the region in $\Real^2$ shown in Figure~\ref{ExMD}(a), with the arrows indicating the velocity of the boundary of the domain.  Notice that the boundary of this domain is initially not smooth for two reasons: it has corners and it intersects itself.  However, during the course of its evolution, the domain changes so that the bounding curve no longer intersects itself, as seen in Figure~\ref{ExMD}(b).  Such an evolution cannot be described by a convecting domain.

\begin{figure}
\centering
\includegraphics[width=2.5in]{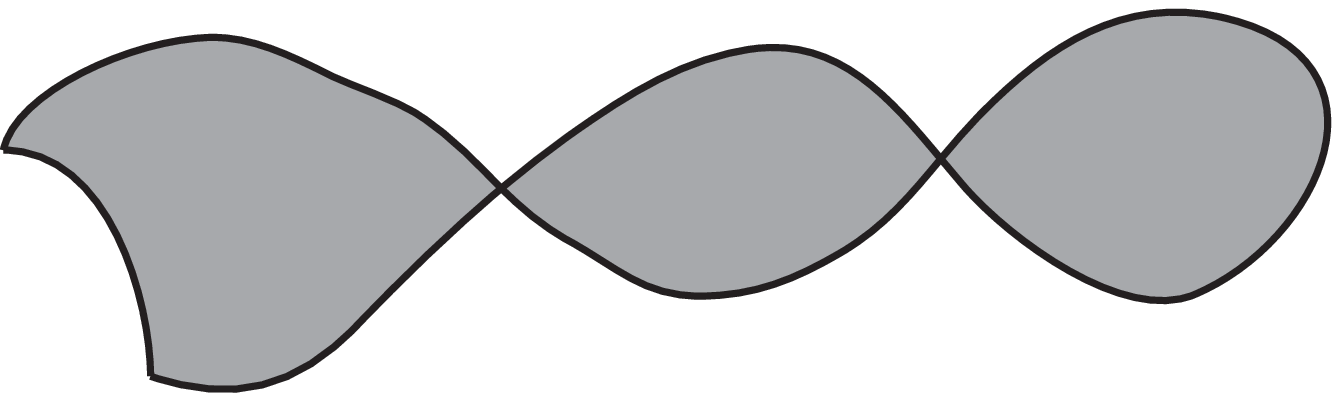}
\hspace{.1in}
\includegraphics[width=2.5in]{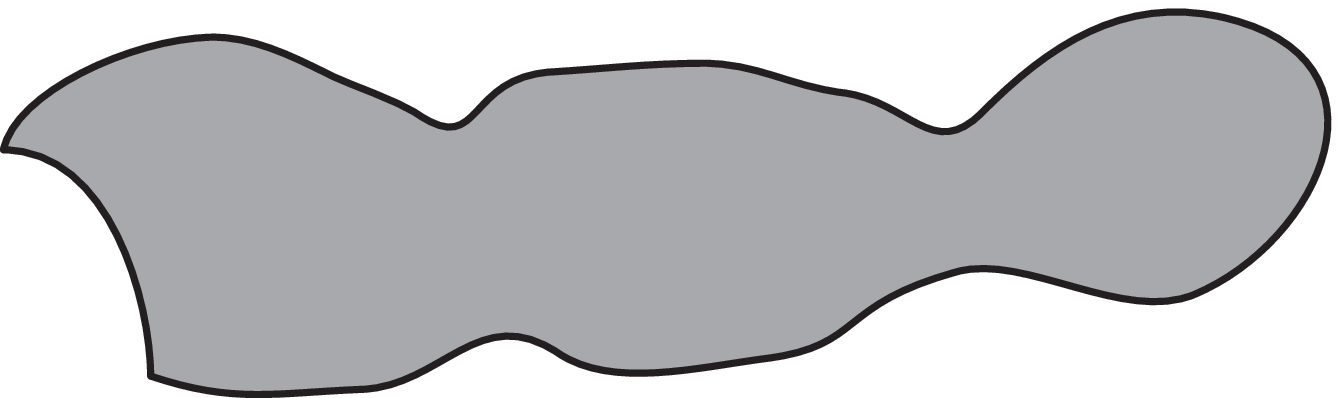}
\thicklines
\put(-315,15){\rotatebox[origin=c]{-50}{$\vector(1,0){12}$}}
\put(-290,31){\rotatebox[origin=c]{-66}{$\vector(1,0){14}$}}
\put(-270.5,31){\rotatebox[origin=c]{-128}{$\vector(1,0){14}$}}
\put(-250,13){\rotatebox[origin=c]{-127}{$\vector(1,0){12}$}}
\put(-295,17){\rotatebox[origin=c]{66}{$\vector(1,0){14}$}}
\put(-276,19){\rotatebox[origin=c]{119}{$\vector(1,0){14}$}}
\put(-280,-10){$(a)$}
\put(-100,-10){$(b)$}
\caption{A simple example of the type of evolving domain allowed in the transport theorem presented here.  During the course of the evolution of the domain, the boundary can transition from intersecting itself, as shown in (a), to not intersecting itself, as seen in (b).}
\label{ExMD}
\end{figure}

The proof of the desired transport theorem uses the ideas of Gurtin, Struthers, and Williams \cite{GSW} in that the evolving domain is considered as a fixed domain in one higher dimension, which can be viewed as space-time.  Moreover, the divergence theorem plays a key role in establishing the result.  Tools from geometric measure theory are used to allow for the consideration of irregular domains.  While the domain is not assumed to be convecting, in some sense the boundary of it is.  

The main motivation for establishing this particular transport theorem is to compute the first variation of a newly introduced fractional length functional.  
The condition of the first variation being zero can be used to motivate a nonlocal notion of curvature for a curve.  This is analogous to how Caffarelli, Roquejoffre, and Savin \cite{CRS10} computed the Euler--Lagrange equation of the fractional perimeter and Abatangelo and Valdinoci \cite{AV14} used it to define a nonlocal mean curvature.  For the work in which this transport theorem is applied to obtain a nonlocal notion of curvature for a curve, see Seguin~\cite{NLL}.  There are other problems in shape optimization theory that do not always involve regular domains, see the work of Bucur and Buttazzo \cite{BB}, and, when considering a perturbation of the domain to compute the first variation, one does not always have a flow map that describes this.  This is exactly the situation that occurred when computing the first variation of the fractional length functional in \cite{NLL}.  Irregular domains are also important in the physical sciences.  See, for example, the works of Marzocchi \cite{27}, Degiovanni, Marzocchi, and Musesti \cite{8}, \u{S}ilhav\'y,\cite{33} and Schuricht \cite{32} one end goal of which is to formulate balance laws for irregular domains. The formulation of such balance laws requires the ability to integrate over irregular domains. Moreover, the balance laws of interest typically include a term involving a time rate of change of an integral in which both the domain and integrand may vary with time, and thus a transport theorem for irregular domains is needed.  Thus, the result established here could have applications outside of its original motivation.

The next section introduces a number of definitions and makes precise the type of evolving domains that will be considered here.  Section~\ref{sectgeo} investigates the geometry of the evolving domains.  In particular, the time-dependent domain will be considered as a fixed domain in one higher dimension and the exterior unit normal to this domain will be determined.  The last section, Section~\ref{sectTT}, contains the proof of the desired transport theorem.

\section{Regularly evolving sets}

We say that a subset $\cS$ of $\Real^d$ is a $k$-dimensional immersed submanifold if it is the range of a function $f\in W^{1,\infty}(\cN,\Real^d)$, where $\cN$ is a $k$-dimensional Riemannian manifold, whose differential $df$ is injective where it exists, which is $\cH^{k}$-a.e..  Such a function $f$ is called an immersion\footnote{Usually an immersion is a smooth function, but here we slighty abuse terminology and apply this term when the function has less regularity.}.  
Let $W^{1,\infty}_\text{im}(\cN,\Real^d)$ denote the space of all immersions in $W^{1,\infty}(\cN,\Real^d)$.  If $f$ is also injective, then $\cS$ is an embedded submanifold of $\Real^d$.  We call the set $\cS$ an \emph{$\cH^k$-embedded submanifold} of $\Real^d$ if it is an immersed submanifold of dimension $k$ and the immersion $f\in W^{1,\infty}_\text{im}(\cN,\Real^d)$ satisfies
\beqn
\cH^k(\{p\in\cS\ |\ \cH^0(f^{-1}(\{p\}))>1\})=0.
\eeqn
This condition says that the set of points of $\cS$ where the image of $\cN$ under $f$ intersects itself has dimension less than $k$.  See Figure~\ref{Haem}.  

We say that a subset $\cA$ is an \emph{$\cH^k$-almost embedded submanifold} if there exists an $\cH^k$-embedded submanifold $\cS$ such that $\cA$ and $\cS$ differ by a set of $\cH^k$-measure zero.  
See Figure~\ref{Haem}. If this is the case, we write
\beqn
\cA\stackrel{\cH^{k}}{=}\cS.
\eeqn
Notice that $\cA$ and $\cS$ have the same tangent space at $\cH^k$-a.e.~point.

\begin{figure}
\centering
\includegraphics[width=3.5in]{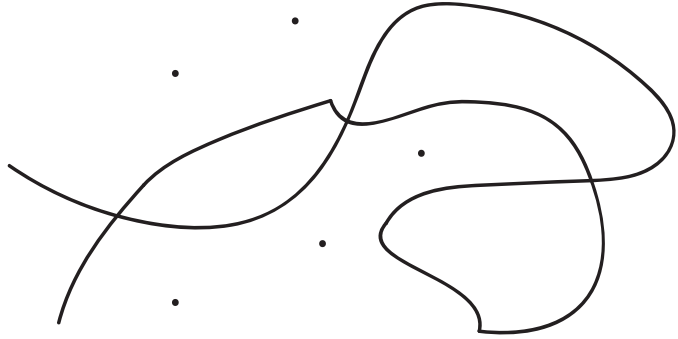}
\thicklines
\caption{The set consisting of just the depicted self-intersecting curve is a $\cH^1$-embedded submanifold of $\Real^2$ since the set is the immersed image of a 1-dimensional Riemannian manifold and the set of intersection points is finite, and hence has $\cH^1$-measure zero.  The set consisting of the self-intersecting curve together with the isolated points represented by dots in the figure is a $\cH^1$-almost embedded submanifold because the isolated points have $\cH^1$-measure zero.}
\label{Haem}
\end{figure}

Let $\cM$ an $m$-dimensional embedded submanifold of $\Real^d$.  For each $t\in \Real$, let $\cO_t$ be an open subset of $\cM$ whose reduced boundary $\partial^* \cO_t$ is an $\cH^{m-1}$-almost embedded submanifold.\footnote{For a definition of reduced boundary see, for example, Ambrosio, Fusco, and Pallara \cite{AFP}.}  Furthermore, assume that each $\partial^* \cO_t$ is related to the image of the same manifold smoothly in time.  More precisely, assume there is a $(m-1)$-dimensional Riemannian manifold $\cN$ and a function $f\in C^1(\Real,W^{1,\infty}_\text{im}(\cN,\Real^d))$ such that
\beqn\label{fprop}
\partial^* \cO_t\stackrel{\cH^{m-1}}{=}f_t(\cN)\quad \text{and}\quad \cH^{m-1}(\{p\in\partial^*\cO_t\ |\ \cH^0(f^{-1}_t(\{p\}))>1\})=0,
\eeqn
where $f_t:=f(t,\cdot)$.  We refer to a family of open sets $\cO_t$ of $\cM$ as a \emph{regularly evolving open set} in $\cM$ if there is a Riemannian manifold $\cN$ and a function $f\in C^1(\Real,W^{1,\infty}_\text{im}(\cN,\Real^d))$ that satisfy \eqref{fprop}.  For simplicity, we refer to $\cO$ as a regularly evolving open set with the understanding that this consists of a collection of open sets indexed by $t$, which we think of as time.

Given a regularly evolving open set $\cO$, the velocity $\bv$ associated with $f$ is defined at $\cH^{m-1}$-a.e.~$p\in\partial^*\cO_t$ by
\beqn\label{vel}
\bv(t,p):=f'(t,f^{-1}_t(p))
\eeqn
for all $t\in\Real$, where the prime denotes partial differentiation with respect to time.  This velocity is not unique as it depends upon $\cN$ and $f$.  However, if $\bn$ denotes the exterior unit-normal to $\cO$ relative to $\cM$, which is defined at $\cH^{n-1}$-a.e.~point of $\partial^*\cO_t$ for all $t\in\Real$, then $V_{\partial}:=\bv\cdot \bn$ is unique $\cH^{n-1}$-a.e..  To see this, consider another $(m-1)$-dimensional Riemannian manifold $\cL$ and a function $g\in C^1(\Real,W^{1,\infty}_\text{im}(\cL,\Real^d))$, such that
\beqn\label{gprop}
\partial^* \cO_t\stackrel{\cH^{m-1}}{=}g_t(\cN)\quad \text{and}\quad \cH^{k-1}(\{p\in\partial^*\cO_t\ |\ \cH^0(g^{-1}_t(\{p\}))>1\})=0.
\eeqn
Given $t\in\Real$ and $p\in\partial^*\cO_t$ such that $f^{-1}_t(p)$ and $g_t^{-1}(p)$ exist, set $z:=g_t^{-1}(p)$ and define $h(t,z):=f^{-1}_t(g_t(z))$.  Notice that this function satisfies
\beqn
g(t,z)=f(t,h(t,z)).
\eeqn
Differentiating this relation with respect to $t$, it follows that the velocity $\bw$ associated with $g$ is related to the velocity $\bv$ associated with $f$ (see \eqref{vel}) by
\beqn
\bw(t,p)=df_t(f^{-1}_t(p)) h'(t,g_t^{-1}(p)) + \bv(t,p).
\eeqn
Since $d f_t(f^{-1}_t(p)) h'(t,g_t^{-1}(p))$ is tangent to $\partial^*\cO_t$, it follows that $\bw\cdot\bn=\bv\cdot\bn=V_{\partial}$, where these quantities are defined.  Since this normal velocity is uniquely defined for a regularly evolving open set, we will say that $\cO$ is a regularly evolving open set in $\cM$ with boundary velocity $V_{\partial}$.

\section{Geometry of regularly evolving sets}\label{sectgeo}

Let $\cO$ be a regularly evolving open set in $\cM$.  Given $t_\circ,t\in \Real$, consider the set
\beqn
\cW:=\{(s,x)\in\Real\times\cM\ |\ s\in(t_\circ,t),\ x\in\cO_s\},
\eeqn
which is an open set in $\cE:=\Real\times\cM$, and hence is a $(m+1)$-dimensional submanifold of $\Real^{d+1}$.  Moreover, the reduced boundary of $\cW$ is given by
\beqn
\partial^* \cW=\cB\cup\cS\cup\cT,
\eeqn
where
\beqn
\cB:=\{t_\circ\}\times\cO_{t_\circ},\quad \cS:=\bigcup_{s\in (t_\circ,t)} \{s\}\times \partial^*\cO_s,\quad\cT:=\{t\}\times \cO_{t}.
\eeqn
One can roughly view $\cB$, $\cS$, and $\cT$ as the bottom, side, and top of $\cW$, respectively.  See Figure~\ref{depicW}.  While it is clear that $\cB$ and $\cT$ are $m$-dimensional manifolds, the fact that $\cS$ is an $\cH^m$-almost embedded submanifold of $\Real^{d+1}$ is not as obvious.  To see that this is the case, consider $f$ and $\cN$ such that \eqref{fprop} holds and define $F:(t_\circ,t)\times\cN\rightarrow\Real^{d+1}$ by
\beqn
F(s,z):=(s,f(s,z))\rfa (s,z)\in(t_\circ,t)\times\cN.
\eeqn
It follows from the properties of $f$ that $F\in W^{1,\infty}_\text{im}((t_\circ,t)\times\cN,\Real^{d+1})$, $\cS\stackrel{\cH^m}{=}g((t_\circ,t)\times\cN)$, and 
\beqn
\cH^{m}(\{p\in\cS\ |\ \cH^0(F^{-1}(\{(t,p)\}))>1\})=0.
\eeqn

The exterior unit-normal to $\cW$ is described in the next result.

\begin{lemma}\label{lemw}
The exterior unit-normal to $\cW$, which exists at $\cH^m$-a.e.~point of $\partial^*\cW$, is given by
\beqn\label{normalW}
\bw:=
\begin{cases}
-\bt & \text{on}\ \cB,\\
\frac{\bn-V_{\partial}\bt}{\sqrt{1+V_{\partial}^2}} & \text{on}\ \cS,\\
\bt & \text{on}\ \cT,
\end{cases}
\eeqn
where $\bt=(t,\bzero)\in\Real\times\Real^d$ is the purely time-like vector, $\bn$ is the exterior unit-normal to $\partial^* \cO$, and $V_\partial$ is the boundary velocity.
\end{lemma}

\begin{proof}
The form of the exterior unit-normal $\bw$ on $\cB$ and $\cT$ is clear from their definitions.  The form of $\bw$ on $\cS$ is characterized by the following properties:
\begin{center}
$\bw-(\bw\cdot\bt)\bt$ is tangent to $\cW$ and orthogonal to $\partial^*\cO$,\\
$\bw\cdot(V_{\partial}\bn+\bt)=0$ and $\bw\cdot\bn>0$,\\
$|\bw|=1$.
\end{center}
The vector $\bw-(\bw\cdot\bt)\bt$ is the projection of $\bw$ onto the plane orthogonal to $\bt$.  To ensure that $\bw$ is orthogonal to $\partial^*\cW$, $\bw-(\bw\cdot\bt)\bt$ must be tangent to $\cW$ and orthogonal to $\partial^*\cO$.  The vector $V_{\partial}\bn+\bt$ is tangent to $\cS$, and so $\bw$ must be orthogonal to it.  The condition $\bw\cdot\bn>0$ ensures that $\bw$ is an exterior normal, as apposed to a interior normal.  The final condition $|\bw|=1$ guarantees that $\bw$ is a unit vector.  The only vector that satisfies all these conditions is that specified in \eqref{normalW}$_2$.
\end{proof}

\begin{figure}
\centering
\includegraphics[width=5in]{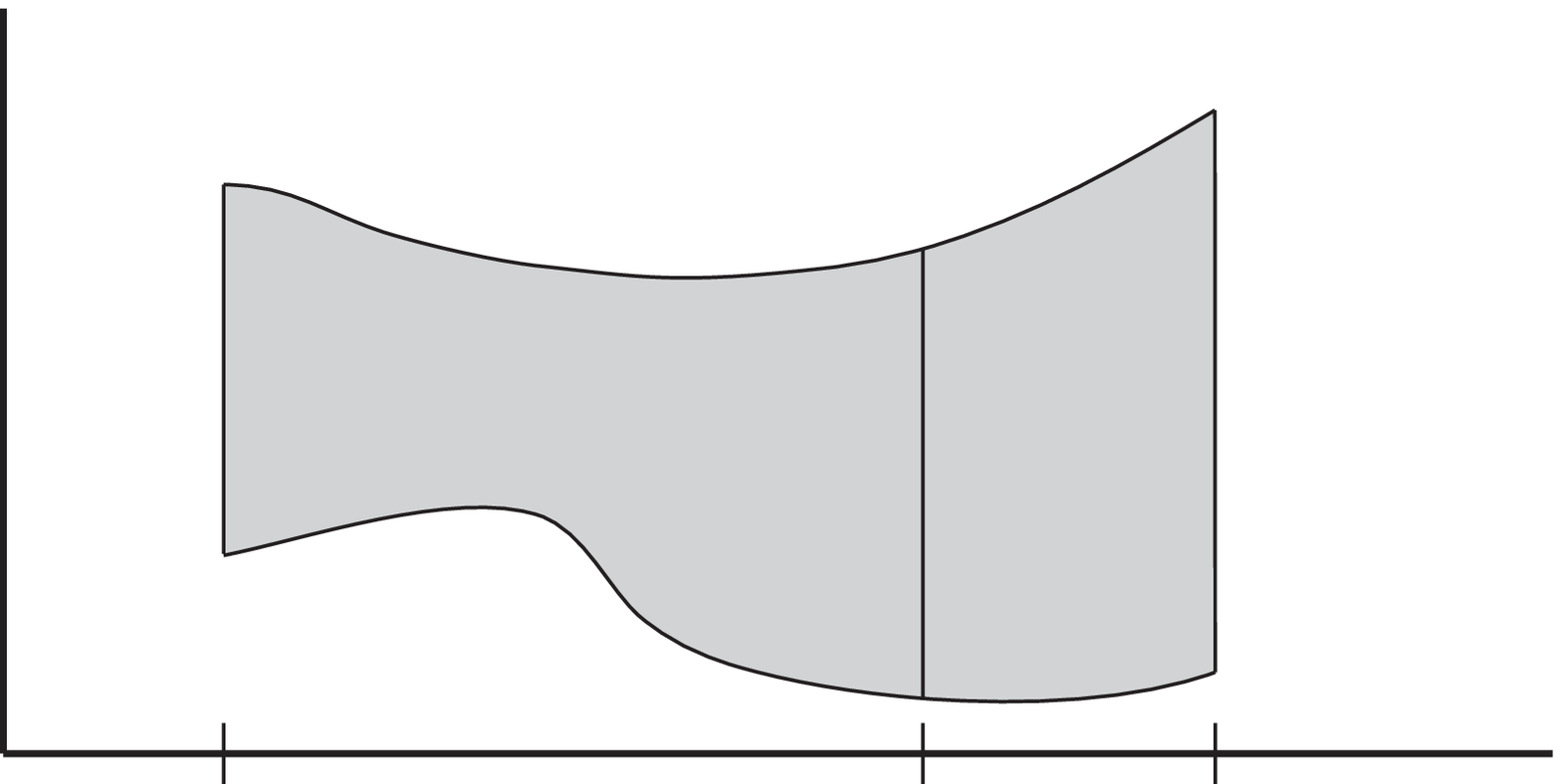}
\thicklines
\put(-358,175){$\cM$}
\put(-320,115){$\cB$}
\put(-75,85){$\cT$}
\put(-255,50){$\cS$}
\put(-255,130){$\cS$}
\put(-70,48){$\bt$}
\put(-78,45){\rotatebox[origin=c]{0}{$\vector(1,0){25}$}}
\put(-325,65){$-\bt$}
\put(-334,65){\rotatebox[origin=c]{180}{$\vector(1,0){25}$}}
\put(-230,35){$\bn$}
\put(-220,30){\rotatebox[origin=c]{-90}{$\vector(1,0){25}$}}
\put(-312,-8){$t_\circ$}
\put(-150,-8){$s$}
\put(-145,70){$\cO_s$}
\put(-81,-8){$t$}
\put(-163.7,130.3){\rotatebox[origin=c]{108}{$\vector(1,0){25}$}}
\put(-163.7,133){$\bw$}
\put(-156.5,132.3){\rotatebox[origin=c]{90}{$\vector(1,0){23}$}}
\put(-150,150.3){$\bw-(\bw\cdot\bt)\bt$}
\put(-200,75){$\cW$}
\caption{This is a depiction of the geometry of $\cW$ in the product space $\Real\times\cM$, which shows its reduced boundary $\partial^*\cW$ consisting of $\cB$, $\cS$, and $\cT$, and its exterior unit-normal $\bw$.}
\label{depicW}
\end{figure}

The following result, which allows us to convert an integral over $\cS$ into an iterated integral, will be useful in proving the main result.

\begin{lemma}\label{lemCV}
If $\phi$ be an integrable real-valued function defined on $\cS$, then
$$\int_\cS \phi\, d\cH^{m} = \int_{t_\circ}^t\int_{\partial^* \cO_s}\phi(s,p)\sqrt{1+V_{\partial}^2(s,p)}\, d\cH^{m-1}(p) ds.$$
\end{lemma}

\begin{proof}
First, find a Riemannian manifold $\cN$ and function $f\in C^1(\Real,W^{1,\infty}_\text{im}(\cN,\Real^d))$ such that \eqref{fprop} holds.  By the area formula \cite{AFP}
\beqn\label{JF-1}
\int_\cS \phi\, d \cH^m = \int_{t_\circ}^t\int_{\cN} \phi(F(s,z))J_F(s,z)\, d\cH^{m-1}(z) ds,
\eeqn
where $J_F=\sqrt{\text{det}( d F^\top d F)}$ is the Jacobian of the function $F:(t_\circ,t)\times\cN\rightarrow \cS$ defined by
\beqn\label{JF0}
F(s,z):=(s,f(s,z))\rfa (s,z)\in (t_\circ,t)\times\cN.
\eeqn
It turns out that 
\beqn
J_F(s,z)=\sqrt{1+V^2_{\partial}(s,f_s(z))}J_{f_s}(z).
\eeqn
To establish this fact, first notice that
\beqn
d F^\top d F=\begin{blockarray}{ccc}
m-1 & 1 &  \\
    \begin{block}{(c|c)c}
      d f^\trans d f & d f^\trans f'   & m-1  \\ 
      \cline{1-2}
      f'^\trans d f  & |f'|^2+1  & 1  \\ 
    \end{block}
\end{blockarray},
\eeqn
where $df$ is the differential of $f$ with respect to the space variable and $f'$ is the derivative with respect to time.  Recall that the determinant of a block matrix can be computed using
\begin{align*}
\det\begin{blockarray}{cc}
  &  \\
    \begin{block}{(c|c)}
      \bA & \bB   \\ 
      \cline{1-2}
      \bC & \bD \Tstrut \\ 
       \end{block}
\end{blockarray}&=\det(\bD)\det(\bA-\bB\bD^{-1}\bC).
\end{align*}
Using this identity, we find that
\beqn\label{JF1}
J_F^2= (|f'|^2+1) \det\Big( d f^\trans d f-\frac{d f^\trans f'\otimes d f^\trans f'}{|f'|^2+1}\Big).
\eeqn
To compute the right-hand side of the previous equation, recall the following identity involving the determinant: if $\bA$ and $\bB$ are linear mappings between inner-product spaces such that the product $\bA^\trans\bB\bA$ makes sense, then
\beqn
\det(\bA^\trans\bB\bA)=\det(\bA^\trans\bA)\det(\bJ^\trans\bB\bJ),
\eeqn
where $\bJ$ is the natural injection of the range of $\bA$ into the domain of $\bB$.  Using this fact, we find that
\beqn\label{JF2}
\text{det}\Big[ d f^\trans d f-\frac{d f^\trans f'\otimes d f^\trans f'}{|f'|^2+1}\Big]=\det(d f^\trans d f)\det\Big(\Big[\textbf{1}_{ f}- \frac{|\bP_{ f} f'|^2}{|f'|^2+1} \be\otimes\be  \Big]\Big),
\eeqn
where $\textbf{1}_{ f}$ is the identity mapping on the range of $d f$, $\bP_{f}$ the projection of $\Real^n$ onto the range of $df$, and $\be:=\bP_f f'/|\bP_f f' |$.  Since $\be$ is a unit vector, we have
\beqn\label{JF3}
\det\Big(\textbf{1}_{ f}- \frac{|\bP_{ f} f'|^2}{|f'|^2+1} \be\otimes\be  \Big)= 1-\frac{|\bP_{ f} f'|^2}{|f'|^2+1}.
\eeqn
Notice that $|f'|^2=|\bP_f f'|^2+|\bn\cdot f'|^2$, and so putting \eqref{JF1}--\eqref{JF3} together yields
\beqn
J_F^2=|f'|^2+1-|\bP_ff'|^2 =V^2_{\partial}+1,
\eeqn
which establishes \eqref{JF0}.  Plugging this result into \eqref{JF-1} and using the area formula again results in
\begin{align}
\int_\cS \phi\, d \cH^m &= \int_{t_1}^t\int_{\cN} \phi(s,f(s,z))\sqrt{1+V^2_{\partial}(s,f_s(z))}J(f_s)(z)\, d\cH^{m-1}(z) ds\\
&= \int_{t_1}^t\int_{\partial^* \cO_s} \phi(s,p)\sqrt{1+V^2_{\partial}(s,p)}\, d\cH^{m-1}(p) ds.
\end{align}
This is the desired result.

\end{proof}


\section{Proof of the transport theorem}\label{sectTT}

Here the transport theorem for regularly evolving open sets is established.  Before this is done, let us recall that $\cE=\Real\times\cM$ is an $(m+1)$-dimensional submanifold of $\Real^{d+1}$.  Given a differentiable vector-field $\ba$ defined on an open subset of $\cE$, we denote the covariant gradient of $\ba$ by $\nabla_\cE\ba$.  
The associated divergence operator is defined by $\text{div}_\cE \ba:= \text{tr}(\nabla_\cE \ba)$.

The set $\cW$, which is an open subset of $\cE$, has finite-perimeter.  Thus, the divergence theorem can be applied.  Namely, if $\ba$ is a smooth tangential vector field defined on $\cW$, we have
\beqn
\int_\cW \text{div}_\cE(\ba) d\cH^m=\int_{\partial^* \cW} \ba\cdot\bw\, \cH^{m-1}.
\eeqn
This result plays an important role in establishing the desired transport theorem.  The fact that the divergence theorem can be used to prove transport theorems is due to Gurtin, Struthers, and Williams \cite{GSW}.

\begin{theorem}
Let $\cO$ be a regularly evolving open set with boundary velocity $V_{\partial}$ and let $\phi\in C^1(\Real,W^{1,1}(\cM,\Real))$.  It follows that
\beqn
\frac{d}{d t}\int_{\cO} \phi\, d\cH^{m}=\int_\cO  \phi'\, d\cH^m+ \int_{\partial^* \cO} \phi V_{\partial}\, d\cH^{m-1},
\eeqn
where prime denotes the partial derivative with respect to time.
\end{theorem}

\begin{proof}
Notice that since $\bt=(1,\textbf{0})$ is a purely time-like vector,
\beqn
\text{div}_\cE(\phi \bt) = \nabla_\cE \phi\cdot \bt = \phi' .
\eeqn
Applying the divergence theorem and then the coarea formula \cite{AFP} results in
\beqn\label{TT1}
\int_{\partial^* \cW} \phi \bt\cdot\bw\, d\cH^{m} = \int_\cW \text{div}_\cE(\phi \bt)\, d\cH^{m+1}=\int_\cW \phi' \, d\cH^{m+1} = \int_{t_\circ}^t \int_{\cO_s} \phi'(s,p) d\cH^m(p) ds.
\eeqn
Using Lemmas \ref{lemw} and \ref{lemCV}, the integral on the left-hand side of \eqref{TT1} can be computed as
\begin{align}
\nonumber\int_{\partial^* \cW} \phi \bt\cdot\bw\, d\cH^{m}&=-\int_\cB \phi \, d\cH^{m}-\int_\cS \frac{\phi V_{\partial }}{\sqrt{1+V^2_{\partial }}} \, d\cH^m+\int_\cT \phi \, d\cH^{m}\\
&\label{TT2}= -\int_{\cO_{t_\circ}}\phi\, d\cH^m-\int_{t_\circ}^t \int_{\partial^* \cO_s}\phi(s,p)V_\partial(s,p)\, d\cH^{m-1}(p) ds +\int_{\cO_{t}}\phi\, d\cH^m.
\end{align}
Substituting \eqref{TT2} into \eqref{TT1} and differentiating with respect to $t$ yields the result.

\end{proof}

\bibliographystyle{acm}
\bibliography{nonlength} 

\end{document}